\RequirePackage{etex}
\RequirePackage{easymat}
\documentclass[12pt]{elsarticle}
\usepackage{amsthm,amsmath,amstext,amssymb,xcolor,scrextend}
\usepackage[all]{xy}
\usepackage[active]{srcltx}

\usepackage{tikz}
\usetikzlibrary{positioning}

\theoremstyle{definition}

\theoremstyle{plain}
 \newtheorem*{thm*}{Theorem}
 
 \newtheorem{lem}{Lemma}
 \newtheorem{prop}{Proposition}
\theoremstyle{remark}

\newtheorem{ex}{Example}

\newcommand{\End}{\mathrm{End}}

\newcommand{\Gl}{\mathrm{GL}}
\newcommand{\Id}{\mathrm{Id}}
\newcommand\Po{\mathcal{S}}

\newcommand\Poe{\widehat{\Po}}
\newcommand\Gr{\mathrm{Gr}}
\newcommand\Mi[1]{C_{#1}}

\newcommand\Ti[1]{\widehat{C_{#1}}}
\newcommand\Mp[1]{\mathbb M_{#1}}
\newcommand\Ma{\mathbb M_n}
\newcommand\dimv{\mathbf{dim}\,}
\newcommand\cdn{\mathbf{cdn}\,}
\newcommand\Q[1]{Q_{#1}}
\newcommand\QT[1]{\widehat{Q}_{#1}}
\newcommand\A[1]{\Pi_{#1}} 
\newcommand\V{\mathbf{V}} 
 
\newcommand\ID{\mathrm{I}} 	
\newcommand\DD[1]{\mathcal D_#1}
\newcommand\Jx{\mathcal{J}_x} 
\newcommand\T{\mathcal{T}} 

\newcommand\R[2]{\mathrm{Fl}_{#2,#1}}

\newcommand\Zn{\mathbb{Z}}

\newcommand\alv{\mathbf{\boldsymbol \alpha}}

\begin{document}
\title{On dimension of poset variety}

\author{Claudia Cavalcante Fonseca}
\ead{ccf@ime.usp.br}
\author{Kostiantyn Iusenko}
\ead{iusenko@ime.usp.br}

\address{Instituto de Matem\'{a}tica e Estat\'{i}stica, Universidade de S\~{a}o Paulo, Brasil}

\begin{abstract}
For a finite partially ordered set  we calculate the dimension  of the variety of its subspace representations having fixed dimension vector. The dimension is given in terms of the Euler quadratic form associated with a partially ordered set, which gives a geometric interpretation of this form.
\end{abstract}

\begin{keyword}
Poset \sep Representations \sep Euler Quadratic form \sep Poset variety.
\MSC  15A63\sep 16G20.
\end{keyword}

\maketitle

{\centering\footnotesize Dedicated to Professor Vladimir Sergeichuk on the occasion of his 70th birthday.\par}

\section{Introduction}

Let $\Po$ be a finite partially ordered set (\textit{poset} in the sequel). All vector spaces in the current text are over a fixed algebraically closed field. Fixing an admissible dimension vector $\alv=(\alpha_0,\alpha_s)_{s\in \Po} \in \mathbb Z^{|\Po|+1}$ (see Preliminaries, for exact definition) consider a variety $\R{\alv}{\Po}$ of tuples of subspaces $(V_s)_{s\in \Po}$ of a vector space 
$V_0$ of dimension $\alpha_0$ such that $V_s\subseteq V_t,$ whenever $s\prec t$ in $\Po$ and $\dim V_s=\alpha_s$ for all $s$ in  $\Po$.
This variety was considered in \cite{futyus} as a set of all subspace representations of $\Po$ having dimension vector $\alv$. The interested reader is referred to \cite{NazarovaRoiter,simson} for information about representations of posets. The variety $\R{\alv}{\Po}$ 
can be viewed as a generalization of several important varieties. For instance: if $\Po$ is an anti-chain, then $\R{\alv}{\Po}$ is a product of Grassmannians; if $\Po$ is linearly ordered, then $\R{\alv}{\Po}$ is a partial flag variety; if $\Po$ is primitive (that is, a disjoint union of linearly ordered sets), then $\R{\alv}{\Po}$ is the multiple flag variety considered in \cite{mwz}. 

In this paper we give the dimension of $\R{\alv}{\Po}$ in terms of the Euler quadratic form $Q_\Po$ associated with $\Po$ (in particular, this answers the question posed in \cite{futyus}). Namely, we prove the following Tits-type equality:
\begin{thm*} 
For any finite poset $\Po$ and admissible vector $\alv=(\alpha_0,\alpha_s)_{s\in \Po} \in \mathbb Z^{|\Po|+1}$ one has
\begin{equation} \label{eqMain}
	\dim \Gl(\alpha_0)-\dim \R{\alv}{\Po}=Q_\Po(\alv).
\end{equation}
\end{thm*}

The importance of integral bilinear forms and integral quadratic forms for classification problems of representations of posets and quivers was first observed in the works of Gabriel \cite{gabriel} and Drozd \cite{drozd} (see also  \cite{bongartz,simsonForms2,simsonForms} and references therein). For instance, the Tits quadratic form associated with a finite quiver has both geometric and homological interpretations (a generalization of the Tits quadratic form for quivers with relations was constructed in \cite{bongartz}). An important  role for applications in representation theory of posets is played by the Tits quadratic form and the Euler quadratic form (see, for instance, \cite{simsonForms2,simsonForms,simson}). The Tits quadratic form associated with a poset is known to have both geometric and homological interpretations 
(see, \cite{drozd}, \cite[Proposition 11.93]{simson} and \cite{bfksy} for an analogue of Tits quadratic form for unitary subspace representations of posets). 
It is well-known that Euler form associated with a poset has a homological interpretation, see, for instance, \cite[Lema 4.1]{simsonForms} and \cite{simsonForms2}. Our result can be viewed as a geometric interpretation of the Euler quadratic form. 

The paper is organized as follows. In Section 2 we give definitions and some preliminary results. In Section 3 we prove the main Theorem. In Section 4 we illustrate the proof of the main result by examples, and prove several consequences of the main result.

\section{Preliminaries}

A finite poset $\Po\equiv(\Po,\preceq)$ is given by the set of elements $\{s_1,\ldots,s_n\}$ and a partial order $\preceq$. Fixing the numbering in $\Po$ 
we make the identifications $\Mp{\Po}(\Zn)\equiv\Ma(\Zn)$ and $\Zn^{\Po}\equiv\Zn^{n}$. The relation $\preceq$ is uniquely defined by the \textit{incidence matrix} $\Mi{\Po}$ of $\Po$; that is, the integral square $n\times n$ matrix 
\begin{equation} \label{incidenceMatrix}
	\Mi{\Po} = [c_{st}]_{s,t \in \Po}\in \Mp{\Po}(\Zn), \quad \mbox{with}\ c_{st}=\left \{ \begin{array}{c}
	   1,\quad \mbox{for}\ s\preceq t,\\
	   0,\quad \mbox{for}\ s\npreceq t.\\ 
	\end{array} \right.
\end{equation}
Given a poset $\Po$, by $\Poe$ we understand its enlargement by unique maximal element $0$; that is, $\Poe\equiv(\Poe,\preceq^0)$ with $\Poe \setminus \{0\}=\Po$ and the order $\preceq^0$ is  obvious. 
 
Recall that the \textit{height} $h(\Po)$ is defined as the cardinality of the longest chain in $\Po$. By $\DD{s}$ denote the following set
$\DD{s}=\{t\in \Po \ | \ t\prec s\}$. The following proposition is a variation of the well-known Theorem (e.g., \cite[Theorem 2]{mir} for the original formulation and the proof).

\begin{prop}\label{partition}
Let $\Po$ be any poset. There are uniquely defined 
subsets $\T_i \subset \Po, i\in [1,h(\Po)]$ such that $\Po = \T_{h(\Po)} \sqcup \dots \sqcup \T_1$ and:
\begin{enumerate}[(i)]
\item for any $i,j$ with $i<j$ and any $t\in \T_i$, there is $s\in \T_j$ such that $t\prec s$ in $\Po$;
\item if $(r,s) \in \T_i \times \T_j$ and $r\prec s$ in $\Po$, then $i<j$. 
\end{enumerate}
\end{prop}

Let $\mathcal X,\mathcal Y$ be any subsets of $\Po$. By
$E_{\mathcal X}$ denote the column matrix $\begin{bmatrix}1 & 1 & \dots & 1 \end{bmatrix}^{tr}$ of size $|\mathcal{X}| \times 1$ and by
$\Mi{\mathcal X,\mathcal Y}$ denote the corresponding restriction of $\Mi{\Po}$; that is, $\Mi{\mathcal X,\mathcal Y}=[c_{st}]$ is an integral $|\mathcal X|\times |\mathcal Y|$ matrix, where $c_{st}$ is defined by \eqref{incidenceMatrix} with $s\in \mathcal X$ and $t\in \mathcal Y$. 
To simplify the notation we denote $\Mi{\mathcal X,\mathcal X}$ by $\Mi{\mathcal X}$, $\Mi{\T_i,\T_j}$ by $\Mi{i,j}$ and $E_{\T_i}$ by $E_i$. In particular, it is easy to see that $\Mi{i}$ is the identity matrix of size $|\T_i|$ (as each $\T_i$ is an antichain in $\Po$) which we denote by $\ID_i$. For any $k\in [1,h(\Po)]$ denote by $\Po_k$ the poset $\T_{h(\Po)} \sqcup \dots \sqcup \T_k$. In particular, $\Po_{1}=\Po$. If $\alv\in \Zn^{\Po}$, by ${\alv}_{\mathcal X}$ we denote the restriction of $\alv$ to $\mathcal X \subseteq \Po$. Again, to simplify the notation, ${\alv}_i$ denotes ${\alv}_{\T_i}$.

We will treat $\Mi{\Po}$ and $\alv\in \Zn^{\Po}$ with respect to the partition of $\Po$ as in Proposition \ref{partition}. That is, 
\begin{equation*}
\Mi{\Po}=
\begin{bmatrix}
\ID_{h(\Po)}& 0 & \dots & 0 \\
\Mi{h(\Po)-1,h(\Po)} & \ID_{h(\Po)-1} & \dots & 0 \\
\vdots &  \vdots & \ddots & \vdots \\
\Mi{1,h(\Po)} & \Mi{1,h(\Po)-1} & \dots &  \ID_{1}
\end{bmatrix},\quad 
\alv=(\alv_{h(\Po)},\dots,\alv_{1}).
\end{equation*}
It is easily verified and well-known that $\Mi{\Po}$ has the following decomposition
\begin{equation*}
\Mi{\Po}=F_{h(\Po)-1}\cdot \ldots \cdot F_{1},
\end{equation*}
in which each $F_{i}, i\in [1,h(\Po)-1]$ is a Frobenious-like matrix having the form
\begin{equation*}
F_i=
\begin{bmatrix}
\ID_{h(\Po)}& 0 & \dots & 0 & 0 & 0 &  \dots & 0 \\
0 & \ID_{h(\Po)-1} & \dots & 0 & 0 & 0 & \dots & 0\\
\vdots &  \vdots & \ddots & \vdots & \vdots & \vdots & \ddots & \vdots\\
0 & 0 & \dots & \ID_{i+1} & 0 & 0 & \dots & 0 \\ 
\Mi{i,h(\Po)} & \Mi{i,h(\Po)-1} & \dots & \Mi{i,i+1}& \ID_{i} & 0 & \dots & 0 \\
0 & 0 & \dots & 0 & 0 & \ID_{i-1 }& \dots & 0 \\
\vdots &  \vdots & \ddots & \vdots & \vdots & \vdots & \ddots & \vdots\\
0 & 0 & \dots &  0 & 0 & 0 & \dots & \ID_{1}
\end{bmatrix}.
\end{equation*}
Therefore 
$\Mi{\Po}^{-1}=F_{1}^{-1}\cdot \ldots \cdot F_{h(\Po)-1}^{-1}$, with
\begin{equation*}
F_i^{-1}=
\begin{bmatrix}
\ID_{h(\Po)}& 0 & \dots & 0 & 0 & 0 &  \dots & 0 \\
0 & \ID_{h(\Po)-1} & \dots & 0 & 0 & 0 & \dots & 0\\
\vdots &  \vdots & \ddots & \vdots & \vdots & \vdots & \ddots & \vdots\\
0 & 0 & \dots & \ID_{i+1} & 0 & 0 & \dots & 0 \\ 
-\Mi{i,h(\Po)} & -\Mi{i,h(\Po)-1} & \dots & -\Mi{i,i+1}& \ID_{i} & 0 & \dots & 0 \\
0 & 0 & \dots & 0 & 0 & \ID_{i-1 }& \dots & 0 \\
\vdots &  \vdots & \ddots & \vdots & \vdots & \vdots & \ddots & \vdots\\
0 & 0 & \dots &  0 & 0 & 0 & \dots & \ID_{1}
\end{bmatrix}.
\end{equation*}

We say that $\alv=(\alpha_0,\alpha_s)_{s\in \Po}\in \Zn^{\Poe}$ is an \textit{admissible dimension vector} if $\alv_{\Po}\cdot \Mi{\Po}^{-1}$ is non-negative vector and $\alpha_0\geq \alpha_s$ for all $s\in \Po$. By $\A{\Po}$ denote the convex cone of all admissible dimension vectors. The following definition will be useful in the sequel. For any $\alv\in \Zn^{\Po}$ and $k\in [1,h(\Po)]$ define 
\begin{equation} \label{iterations}
\alv^{(k)}:=\alv^{(k-1)}\cdot F_{k-1}^{-1}, \quad \mbox{with}\quad \alv^{(1)}:=\alv. 
\end{equation}
Hence, it follows that 
\begin{equation} \label{iterCartan}
\alv\cdot \Mi{\Po}^{-1}=\alv\cdot F_{1}^{-1}\cdot \ldots \cdot F_{h(\Po)-1}^{-1}=\alv^{(h(\Po))},
\end{equation}
and  that $\alv_{\Po}\in \A{\Po}$ iff all vectors $\alv^{(k)}$ are non-negative.

We denote by $\QT{\Po},\Q{\Po}:\Zn^{\Poe} \to \Zn$  the Tits quadratic form, and the Euler quadratic form of ${\Poe}$ defined by the formula
\begin{equation*}
\begin{split}
\QT{\Po}(\alv)&=\alv \cdot \Ti{\Po} \cdot \alv^{tr}, \\
\Q{\Po}(\alv)&=\alv \cdot \Mi{\Poe}^{-1} \cdot \alv^{tr},
\end{split}
\end{equation*}
respectively, where 
$\Ti{\Po}=\left[
\begin{array}{c|c}
  1 & 0 \\ \hline
  -E_\Po & \Mi{\Po}^{tr}
\end{array}
\right]\in \Mp{\Poe}(\Zn)$
is the Tits matrix of $\Poe$. The reader is referred to \cite{simsonForms,simson} (and references therein) for a detailed study of quadratic and bilinear forms associated with posets and their applications for classification problems of representations of posets.

\section{Proof of the Theorem}

First we prove two auxiliary lemmas. Denote by $d_{\Po}$ the function $\R{\alv}{\Po}\to \mathbb Z$ which maps each point $(V_s)_{s\in \Po}\in \R{\alv}{\Po}$ to $\dim\Big (\sum_{y\in \T_{h(\Po)}} V_y\Big )$. It is easy to check that $d_{\Po}$ is semi-continuous. Moreover we have the following 

\begin{lem} \label{sum}
Let $\alv\in \Zn^{\Poe}$ be a vector such that $\alv\cdot \Mi{\Poe}^{-1}$ is non-negative.
There is an open subset in $\R{\alv}{\Po}$ where $d_{\Po}$ attains maximum equal to 
$\alv_{\Po}\cdot \Mi{\Po}^{-1}\cdot E_\Po$. 
\end{lem}

\begin{proof}
First suppose that $h(\Po)=1$. In this case $\R{\alv}{\Po}=\prod_{s\in \Po} \Gr(\alpha_s,\alpha_0)$. As  $\alv\cdot \Mi{\Poe}^{-1}$ is non-negative for $\alv=(\alpha_0;\alpha_s)_{s\in \Po}$, we have that $\sum_{s\in \Po} \alpha_s\leq \alpha_0$. Hence for general point $(V_s)_{s\in \Po}$ in $\R{\alv}{\Po}$ the sum $\sum_{s\in \Po} V_s$ is direct, therefore $\dim(\sum_{s\in \Po} V_s)=\sum_{s\in \Po}\alpha_{s}$ is a maximal value of $d_\Po$, which is equivalent to statement. 

Assume that $h(\Po)>1$. Consider the natural map $\pi:\R{\alv}{\Po} \to \prod_{s\in \T_1} \Gr(\alpha_s,\alpha_0)$ which maps the tuple $(V_s)_{s\in \Po}$ to $(V_s)_{s\in \T_1}$. The map $\pi$ is obviously surjective (as $\T_1$ is a set of minimal points in $\Po$).
Recall that $\Po_2=\T_{h(\Po)}\sqcup\dots \sqcup \T_{2}$. Abusing the notation let $\alv^{(2)}$ be a restriction of $\alv^{(2)}$ (defined by \eqref{iterations}) to $\Po_2$. Consider the variety  $\R{\alv^{(2)}}{\Po_2}$ and the corresponding function $d_{\Po_2}$. Analyzing the fibers of the map $\pi$ it is straightforward to check that 
$$
	\max d_{\Po}=\max d_{\Po_2}+\alv_1\cdot E_1.
$$
Therefore, proceeding inductively, we get
$$
\max d_{\Po}=\sum_{i=1}^{h(S)} \alv_i^{(i)} \cdot E_i=\alv^{(h(S))}_{\Po}\cdot E_{\Po}.
$$
By \eqref{iterCartan} we have $\alv^{(h(S))}\cdot E_{\Po}=\alv_{\Po}\cdot \Mi{\Po}^{-1}\cdot E_\Po$, thus the statement follows. \end{proof}

\begin{lem} \label{lem_alg}
Let $\Po$ be a poset, $x$ be a maximal element in $\Po$ and $\alv \in \Zn^{\Poe}$. Then $$\Q{\Po}(\alv) - \Q{\Po\setminus\{x\}}(\alv_{\Po\setminus\{x\}}) = - (\alpha_x-{\alv}_{\DD{x}}\cdot\Mi{\DD{x}}^{\,-1}\cdot{E_{\DD{x}}})(\alpha_0-\alpha_x).$$ 
\end{lem}
\begin{proof}
If $h(\Po)=1$, then the statement is clear (since $\DD{x}$ is empty and $\Mi{\Po}^{\,-1} =\Id$).
Suppose $h(\Po)\geq 2$. Write $\Po$ as $\Jx \cup \{x\} \cup \DD{x}$ and 
$\alv$ as $(\alv_0;\alv_{\Jx},\alpha_x, \alv_{\DD{x}})$, where $\Jx=\Po \setminus (\DD{x} \cup \{x\})$.   The incidence matrices of $\Poe$ and $\Poe\setminus\{x\}$ (with respect to above decomposition of $\Po$) have the form
\begin{equation*}
  \Mi{\Poe} = \begin{bmatrix} 
    			1 & 0 & 0 & 0\\
                E_{\Jx} & \Mi{\Jx} & 0 & 0\\
                1 & 0 & 1 & 0\\
                E_{\DD{x}} & \Mi{\DD{x},\Jx} & E_{\DD{x}} & \Mi{\DD{x}}
    		   \end{bmatrix}, \quad 
\Mi{\Poe\setminus\{x\}} = \begin{bmatrix} 
    			1 & 0 & 0 \\
                E_{\Jx} & \Mi{\Jx} & 0\\
                E_{\DD{x}} & \Mi{\DD{x},\Jx} & \Mi{\DD{x}}
    		   \end{bmatrix}.     
\end{equation*}
Therefore 
\begin{equation*}
\begin{split}
  \Mi{\Poe}^{\,-1}&=\begin{bmatrix} 
    			1 & 0 & 0 & 0\\
                -\Mi{\Jx}^{\,-1}E_{\Jx} & 
                \Mi{\Jx}^{\,-1} & 
                0 & 0\\
                -1 & 0 & 1 & 0\\
                \Mi{\DD{x}}^{\,-1}\Mi{\DD{x},\Jx}\Mi{\Jx}^{\,-1}E_{\Jx} & 
                 -\Mi{\DD{x}}^{\,-1}\Mi{\DD{x},\Jx}\Mi{\Jx}^{\,-1}& 
                -\Mi{\DD{x}}^{\,-1}E_{\DD{x}}& 
                \Mi{\DD{x}}^{\,-1}
    		   \end{bmatrix},\\
\Mi{\Poe\setminus\{x\}}^{\,-1}&=\begin{bmatrix} 
    			1 & 0 & 0 \\
                -\Mi{\Jx}^{\,-1}E_{\Jx} & 
                \Mi{\Jx}^{\,-1} & 0 \\
                \Mi{\DD{x}}^{\,-1}\Mi{\DD{x},\Jx}\Mi{\Jx}^{\,-1}E_{\Jx} - \Mi{\DD{x}}^{\,-1}E_{\DD{x}} & 
                -\Mi{\DD{x}}^{\,-1}\Mi{\DD{x},\Jx}\Mi{\Jx}^{\,-1}&         
                \Mi{\DD{x}}^{\,-1}
    		   \end{bmatrix}.               
\end{split}
\end{equation*}
Hence the corresponding Euler forms are:
\begin{equation*}
\begin{split}
\Q{\Po}(\alv) &=(\alv_0;\alv_{\Jx},\alpha_x, \alv_{\DD{x}})\cdot \Mi{\Poe}^{\,-1} \cdot (\alv_0;\alv_{\Jx},\alpha_x, \alv_{\DD{x}})^{tr}\\
&= (\alpha_0 - \alv_{\Jx}\cdot\Mi{\Jx}^{\,-1}\cdot E_{\Jx}-\alpha_x + \alv_{\DD{x}}\cdot \Mi{\DD{x}}^{\,-1}\cdot\Mi{\DD{x},\Jx}\cdot\Mi{\Jx}^{\,-1})\alpha_0
\\&\qquad+(\alv_{\Jx}\cdot\Mi{\Jx}^{\,-1} - \alv_{\DD{x}}\cdot\Mi{\DD{x}}^{\,-1}\cdot\Mi{\DD{x},\Jx}\cdot\Mi{\Jx}^{\,-1})\cdot\alv_{\Jx}^{tr}
\\&\qquad+(\alpha_x - \alv_{\DD{x}}\cdot\Mi{\DD{x}}^{\,-1}\cdot E_{\DD{x}})\alpha_x+\alv_{\DD{x}}\cdot\Mi{\DD{x}}^{\,-1}\cdot\alv_{\DD{x}}^{tr},
\end{split}
\end{equation*}
and
\begin{equation*}
\begin{split}
&\Q{\Po\setminus\{x\}}(\alv_{\Po\setminus\{x\}})= (\alv_0;\alv_{\Jx},\alv_{\DD{x}})\cdot \Mi{\Poe\setminus\{x\}}^{\,-1} \cdot (\alv_0;\alv_{\Jx}, \alv_{\DD{x}})^{tr}\\ 
&\quad= (\alpha_0 - \alv_{\Jx}\cdot\Mi{\Jx}^{\,-1}\cdot E_{\Jx}+ \alv_{\DD{x}}\cdot\Mi{\DD{x}}^{\,-1}\cdot\Mi{\DD{x},\Jx}\cdot\Mi{\Jx}^{\,-1} -\alv_{\DD{x}}\cdot\Mi{\DD{x}}^{-1}\cdot E_{\DD{x}})\alpha_0
\\&\qquad +(\alv_{\Jx}\cdot\Mi{\Jx}^{\,-1} - \alv_{\DD{x}}\cdot\Mi{\DD{x}}^{\,-1}\cdot\Mi{\DD{x},\Jx}\cdot\Mi{\Jx}^{\,-1})\cdot\alv_{\Jx}^{tr}+\alv_{\DD{x}}\cdot\Mi{\DD{x}}^{\,-1}\cdot\alv_{\DD{x}}^{tr}.
\end{split}
\end{equation*}
Thus, 
\begin{equation*}
\begin{split}
\Q{\Po}(\alv) - \Q{\Po\setminus\{x\}}(\alv_{\Po\setminus\{x\}}) &= -\alpha_x\alpha_0 + \alv_{\DD{x}}\cdot\Mi{\DD{x}}^{-1}\cdot E_{\DD{x}}\alpha_0 + \alpha_x^2-\alv_{\DD{x}}\cdot\Mi{\DD{x}}^{\,-1}\cdot E_{\DD{x}}\alpha_x 
\\&= - (\alpha_x-{\alv}_{\DD{x}}\cdot\Mi{\DD{x}}^{\,-1}\cdot{E_{\DD{x}}})(\alpha_0-\alpha_x).
\end{split}
\end{equation*}
\end{proof}

Now we prove the main Theorem. We proceed by induction on cardinality of poset.
If $\Po=\{x\}$, then $\dim \R{\alv}{\Po} = \dim \Gr(\alpha_x,\alpha_0) = \alpha_x(\alpha_0-\alpha_x) $. On the other hand, 
    \begin{eqnarray*}
    	\dim \Gl(\alpha_0)-\Q{\Po}(\alv) &=& {\alpha_0}^2 -\begin{bmatrix}\alpha_0 &\alpha_x \end{bmatrix}\cdot \begin{bmatrix}1 & 0\\ 1 & 1\end{bmatrix}^{\,-1} \cdot\begin{bmatrix}\alpha_0 \\\alpha_x \end{bmatrix} 
    \\&=& {\alpha_0}^2 - (\alpha_0 - \alpha_x)\alpha_0 - {\alpha_x}^2= \alpha_x(\alpha_0-\alpha_x).
\end{eqnarray*}
Suppose that $|\Po|>1$. Let $x\in \Po$ be a maximal element. Consider the natural map $\pi:\R{\alv}{\Po} \to \R{\alv_{\Po\setminus\{x\}}}{\Po \setminus \{x\}}$, which maps a tuple $(V_s)_{s\in \Po}$ to $(V_s)_{s\in \Po\setminus\{x\}}$. Map $\pi$ is surjective, as $x$ is maximal. Generic fiber of $\pi$ has the form $$\Gr(\alpha_x-X,\alpha_0-X),$$ in which $X$ is a general dimension of the sum of subspaces $V_y$ over all $y\to x$ in $\Po$, and we have
\begin{equation} \label{equalityT1}
\dim \R{\alv}{\Po} =  
\dim \Gr(\alpha_x- X,\alpha_0-X)+\dim\R{\alv_{\Po\setminus \{x\}}}{\Po \setminus \{x\}}.
\end{equation}
Applying Lemma \ref{sum} for the poset $\DD{x}$ and dimension vector 
${\alv}_{\DD{x}}$, we have that $X={\alv}_{\DD{x}}\Mi{\DD{x}}^{\,-1}{E_{\DD{x}}}$. Therefore
\begin{equation} \label{equalityT2}
	\dim \Gr(\alpha_x-X,\alpha_0-X)=(\alpha_x-{\alv}_{\DD{x}}\cdot\Mi{\DD{x}}^{\,-1}\cdot{E_{\DD{x}}})(\alpha_0-\alpha_x).
\end{equation}
By induction hypothesis we have
\begin{equation}\label{equalityT3}
\dim \R{\alv_{\Po\setminus \{x\}}}{\Po \setminus \{x\}}=\dim \Gl(\alpha_0)-\Q{\Po\setminus\{x\}}(\alv_{\Po\setminus\{x\}}).
\end{equation}
Hence (combining \eqref{equalityT1}-\eqref{equalityT3}), 
\begin{equation*}
\dim \Gl(\alpha_0)-\dim \R{\alv}{\Po}=\Q{\Po\setminus\{x\}}(\alv_{\Po\setminus\{x\}})-(\alpha_x-{\alv}_{\DD{x}}\cdot\Mi{\DD{x}}^{\,-1}\cdot{E_{\DD{x}}})(\alpha_0-\alpha_x).
\end{equation*}
By Lemma \ref{lem_alg} the right hand side of last equality equals $\Q{\Po}(\alv)$ and we are done. 

\section{Examples and consequences}

\begin{ex}
We will illustrate the steps in the proof of the main Theorem on the following example. Consider the poset $\Po=\{1,2,3,4,5,6,7\}$ with the order $1\prec 3, 1\prec 4, 1\prec 5, 2\prec 4, 2 \prec 5, 3\prec 6, 3\prec 7, 4\prec 6, 4 \prec 7, 5\prec 7$ and the dimension vector $\alv=(8;1,2,2,4,5,6,7)$. 
Element $6\in \Po$ is maximal with $\alpha_6=6$. Calculating $X$ as in the proof of the main Theorem we get $X=5$ (as $X$ is a general sum of subspaces with dimension $\alpha_3=2$, $\alpha_4=4$ having common subspace of dimension $1$). In this case the equality \eqref{equalityT1} has the following graphical form (where the components of dimension vectors we place at corresponding vertices of Hasse diagrams):
\[ 
\xymatrix@=0.5pc{
&& {8} &&&&&&& {8} 
\\ \\
& {7}  \ar@{-}[uur]  && {6} \ar@{-}[uul] &&&&&&  {7} \ar@{-}[uu]
\\ \\
{5} \ar@{-}[uur] && {4} \ar@{-}[uul] \ar@{-}[uur] && {2} \ar@{-}[uul] \ar@{-}[uulll] 
& = &  {8-5} & +  & {5}  \ar@{-}[uur]  && {4} \ar@{-}[uul]  && {2} \ar@{-}[uulll]
 \\ \\
& {2} \ar@{-}[uul] \ar@{-}[uur] && {1} \ar@{-}[uur] \ar@{-}[uul] \ar@{-}[uulll]  & & & 6-5 \ar@{-}[uu] &&& {2} \ar@{-}[uul] \ar@{-}[uur]  && {1} \ar@{-}[uur] \ar@{-}[uul] \ar@{-}[uulll]
}
\]  
The induction hypothesis yields  
$$\dim \R{\alv_{\Po\setminus \{6\}}}{\Po\setminus \{6\}}=8^2-\Q{{\Po\setminus \{6\}}}(\alv_{\Po\setminus \{6\}})=35.$$ On the other hand, $\dim \Gr(6-5,8-5)=2$. Therefore by \eqref{equalityT1} $\dim \R{\alv}{\Po}=37$ which is exactly the value of $8^2-\Q{\Po}(\alv)$.
\end{ex}
The next example shows that the assumption $\alv$ to be admissible in Theorem is necessary.
\begin{ex}
Consider the poset $\Po=\{1,2,3,4\}$ with the order $1\prec 3, 1\prec 4, 2\prec 3, 2\prec 4$ and the dimension vector $\alv=(4;2,2,3,3)$. That is, the Hasse diagram of $\Poe$ with the integers $\alpha_s, s\in \Poe$  placed at the corresponding vertices, has the following form:
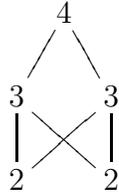
\[
\xymatrix@=0.5pc{ & {4}  &&
 \\ \\
{3} \ar@{-}[uur]   && {3} \ar@{-}[uul]
 &&
 \\ \\
{2} \ar@{-}[uu] \ar@{-}[uurr]  && {2} \ar@{-}[uu] \ar@{-}[uull] &&  }
\] 
The dimension vector $\alv$ is not admissible. One checks that the dimension of of $\R{\alv}{\Po}$ is 7, while $4^2-\Q{\Po}(\alv)=8$. 

In this case we can calculate the dimension of $\R{\alv}{\Po}$ via the Euler form of enlarged poset. Namely, let $\widetilde \Po=\Po\cup\{5\}$, where $5$ is a minimal point, and let 
$\widetilde \alv=(4;1,2,2,3,3)$. 
That is, the corresponding Hasse diagram is 
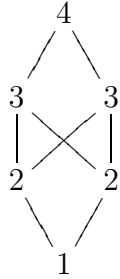
\[
\xymatrix@=0.5pc{ & {4}  &&
 \\ \\
{3} \ar@{-}[uur]   && {3} \ar@{-}[uul]
 &&
 \\ \\
{2} \ar@{-}[uu] \ar@{-}[uurr]  && {2} \ar@{-}[uu] \ar@{-}[uull] && 
\\ \\ 
& {1} \ar@{-}[uur] \ar@{-}[uul] &&
}
\] 
For any point $(V_s)_{s\in \Po}\in \R{\alv}{\Po}$ the intersection between $V_1$ and $V_2$ contains a one dimensional subspace, therefore $\dim \R{\alv}{\Po}=\dim\R{\widetilde \alv}{\widetilde \Po}$. 
As $\widetilde \alv$ is an admissible dimension vector for $\widetilde \Po$, we have that $\dim\R{\widetilde \alv}{\widetilde \Po}=4^2-\Q{\widetilde \Po}(\widetilde \alv)=7$.
\end{ex}
Nevertheless there are non-admissible vector such that \eqref{eqMain} holds.
\begin{ex}
Consider the poset $\Po=\{1,2,3\}$ with the order $1\prec 3, 2\prec 3$ and the dimension vector $\alv=(4;2,2,3)$. That is, the Hasse diagram of $\Poe$ has the form:
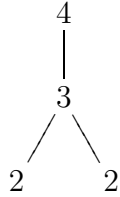
\[
\xymatrix@=0.5pc{ & {4}  &
 \\ \\
&{3} \ar@{-}[uu]  &
 \\ \\
{2}  \ar@{-}[uur]  && {2}  \ar@{-}[uul] &&  }
\] 
The dimension vector $\alv$ is not admissible. In the contrast to the previous example, in this case 
$$\dim \R{\alv}{\Po}=7=4^2-\Q{\Po}(\alv).$$
\end{ex}

We show several consequences of equation \eqref{eqMain} for representations of posets. Recall that a \textit{subspace representation} of $\Po$ is a tuple $\V=(V_0;V_s)_{s\in \Po}$ of vector subspaces 
$V_s$ of a vector space $V_0$ such that $V_s\subseteq V_t,$ whenever $s\prec t$ in $\Po$ (for more details see \citep[Chapter 5]{simson}). The \textit{dimension vector} of $\V$ is defined as $\dimv\, \V=(\dim V_0,\dim V_s)_{s\in \Po}$. The variety $\R{\alv}{\Po}$ can be viewed as a set of all subspace representations of $\Po$ having the dimension vector $\alv$. 
The group $\Gl(\alpha_0)$ acts on $\R{\alv}{\Po}$  via the base change so that the orbits of this action are in a bijection with the isomorphisms classes of subspace
representations of $\Po$ with the dimension $\alv$ (the reader is refered to \cite{futyus}, where the authors studied the corresponding moduli space of this action).

Following \cite{mwz} we say that an admissible dimension vector $\alv$ is of \textit{finite type} if the group $\Gl(\alpha_0)$ has only finitely many orbits on variety $\R{\alv}{\Po}$. We say that a non-zero $\alv' \in \A{\Po}$ is a \textit{summand} of $\alv \in \A{\Po}$ if 
$\alv-\alv'\in \A{\Po}$. It follows
from the Krull-Schmidt theorem that if $\alv$ is of finite type, then every summand of
$\alv$ is also of finite type. Theorem 2.2 from \cite{mwz} can be treated as a classification of dimensions of finite type for primitive posets (this result was generalized in \cite{drozdKub} for any poset). It trivially follows that if ${\Po}$ is a poset of finite representation type (that is, there exist only finitely many of indecomposable subspace representations of $\Po$), then any admissible dimension vector $\alv$ is of finite type. Classification of posets of finite representation type is given in \cite{kle}.

\begin{prop}
Let $\alv$ be an admissible vector of finite type so that $Q_{\Po}(\alv)=1$, then there exists a Schurian representation
of $\Po$ with dimension vector $\alv$.
\end{prop}

\begin{proof}
Recall that a representation $\V$ is called Schurian if $\End(\V)$ is one dimensional. As $\alv$ is of finite type, hence the variety $\R{\alv}{\Po}$ has a dense open (in Zarisky topology) $\Gl(\alpha_0)$ orbit $\mathcal O$. Let $\V$ be any representative from $\mathcal O$, hence 
$$
\dim \End(\V)=\dim \textrm{Stab}_{\Gl(\alpha_0)}\V=\dim \Gl(\alpha_0)-\dim \R{\alv}{\Po}=Q_{\Po}(\alv).
$$
Therefore $\V$ is Schurian.
\end{proof}

\begin{prop}
If dimension vector $\alv$ is of finite type, then 
$Q_{\Po}(\alv')\geq 1$ for any summand $\alv'$ of $\alv$.
\end{prop}
\begin{proof}
As the one-dimensional subgroup of scalar matrices in $\Gl(\alpha_0)$ acts trivially on $\R{\alv}{\Po}$, hence
$\dim \Gl(\alpha_0)-1\geq \dim \R{\alv}{\Po}$. Therefore $Q_{\Po}(\alv)\geq 1$. Any summand of $\alv'$ of $\alv$ is of finite type as well, therefore $Q_{\Po}(\alv')\geq 1$.
\end{proof}
The \textit{coordinate dimension vector} of $\V$ is defined as $\cdn\, \V=(c_0,c_s)_{s\in \Po}$, where $c_0=\dim V_0$, $c_s=\dim (V_s/\sum_{t\prec s} V_t)$. A representation $\V$ is said to be \textit{coordinate} if for any $s\in \Po$ the sum $\sum_{t\to s} V_t$ is direct. If $\alv\in \A{\Po}$ is of finite type, then a general representation of $\Po$ of dimension $\alv$ is coordinate. Using this fact, and taking into account the relation between $\QT{\Po}$ and $\Q{\Po}$ via the relation between $\cdn$ and $\dimv$ (see Preliminaries, and \cite[Proposition 1.1]{futyus}) we have the following analogue of Theorem 1 from \cite{drozd} formulated in terms of variety $\R{\alv}{\Po}$ and the Euler quadratic form $\Q{\Po}(\alv)$.
\newpage
\begin{prop}
Suppose that $\Q{\Po}(\alv)>0$ for any non-zero $\alv \in \A{\Po}$. Then:
\begin{itemize}
\item[(a)] the poset $\Po$ is of finite representation type;
\item[(b)] if $\Q{\Po}(\alv)=1$, then there exists unique (up to isomorphism) indecomposable subspace representation of $\Po$ having dimension $\alv$;
\item[(c)] if $\Q{\Po}(\alv)>1$, then there is no indecomposable subspace representation with dimension vector $\alv$. 
\end{itemize}
\end{prop}

\section*{Acknowledgement}
Part of this work was done during the visit of K.I. to Syracuse University in 2017; he would like to thank Mark Kleiner for the hospitality, very helpful discussions and comments. K.I. is supported in part by Fapesp grants 2014/09310-5,
2015/00116-4 and by CNPq grant 456698/2014-0.

\end{document}